\documentclass[a4paper,11pt]{article}
\usepackage{graphicx}
\usepackage{amssymb}
\usepackage{amsmath}
\usepackage{mathtools}
\usepackage{amsfonts}
\usepackage{appendix}
\usepackage{color}
\usepackage{multirow}
\usepackage{pdflscape}
\usepackage{xcolor}
\def\tr{\mathop{\rm tr}\nolimits}

\def\tr{\mathop{\rm tr}\nolimits}

\def\vec{\mathop{\rm vec}\nolimits}

\newcommand{\bmu}{\boldsymbol{\mu}}

\newcommand{\0}{\boldsymbol{0}}

\newcommand{\bSigma}{\boldsymbol{\Sigma}}

\newcommand{\bB}{\boldsymbol{B}}

\newcommand{\bA}{\boldsymbol{A}}

\newcommand{\bd}{\boldsymbol{d}}

\newcommand{\be}{\boldsymbol{e}}

\newcommand{\bI}{\boldsymbol{I}}

\newcommand{\bp}{\boldsymbol{p}}

\newcommand{\bq}{\boldsymbol{q}}

\newcommand{\bu}{\boldsymbol{u}}

\newcommand{\bx}{\boldsymbol{x}}
\newcommand{\bX}{\boldsymbol{X}}
\newcommand{\by}{\boldsymbol{y}}

\newcommand{\calE}{\mathcal{E}}

\newcommand{\calL}{\mathcal{L}}

\newcommand{\calN}{\mathcal{N}}

\newcommand{\calR}{\mathcal{R}}

\newcommand{\calT}{\mathcal{T}}

\newcommand{\calW}{\mathcal{W}}

\newcommand{\A}{\boldsymbol{\mathcal{A}}}

\renewcommand{\S}{\boldsymbol{\mathcal{S}}}
\newcommand{\T}{\boldsymbol{\mathcal{T}}}

\newcommand{\X}{\boldsymbol{\mathcal{X}}}

\def\tr{\mathop{\rm tr}\nolimits}

\def\tr{\mathop{\rm tr}\nolimits}

\def\vec{\mathop{\rm vec}\nolimits}

\newtheorem{theorem}{Theorem}

\newtheorem{definition}{Definition}

\newtheorem{remark}{Remark}

\newenvironment{proof}[1][Proof]{\textbf{#1.} }{\ \rule{0.5em}{0.5em}}

\begin{document}

\title{\bf Some theoretical results on tensor elliptical distribution}

\bigskip

\author{{ M. Arashi}
\vspace{.5cm} \\\it
 \it Department of Statistics, School of Mathematical Sciences\\\vspace{.5cm} \it Shahrood University of Technology, Shahrood, Iran
 }

\date{}
\maketitle

\begin{quotation}
\noindent {\it Abstract:}
The multilinear normal distribution is a widely used tool in tensor analysis of magnetic resonance imaging (MRI). Diffusion tensor MRI provides a statistical estimate of a symmetric $2^{\textnormal{nd}}$-order diffusion tensor, for each voxel within an imaging volume.
In this article, tensor elliptical (TE) distribution is introduced as an extension to the multilinear normal (MLN) distribution. Some properties including the characteristic function and distribution of affine transformations are given. An integral representation connecting densities of TE and MLN distributions is exhibited that is used in deriving the expectation of any measurable function of a TE variate.
\par

\vspace{9pt} \noindent {\it Key words and phrases:} Characteristic generator; Inverse Laplace transform; Stochastic representation; Tensor; Vectorial operator.

\par

\vspace{9pt} \noindent {\it AMS Classification:} Primary: 62E15, 60E10 Secondary: 53A45, 15A69\par

\end{quotation}\par

\section{Introduction}
Nowadays, analysis of matrix-valued data sets is become quite common in medical sciences, since the collected data are of multiple-way (multiple-component) arrays. For example in medical imaging, it has become possible to collect magnetic resonance imaging (MRI) data that can be used to infer the apparent diffusivity of water in tissue in vivo. In this regard, there is a need to consider parallel extensions of bilinear forms\footnote{Bilinear form is a two-way (two-component) array, with each component represents a vector of observations}, namely tensor matrices.
Tensor matrices have been commonly used to approximate the diffusivity profile of images. This approximation yields a diffusion tensor magnetic resonance imaging (DT-MRI) data set. Processing of DT-MRI data sets has scientific significance in clinical sciences. Figure \ref{brain1} shows the tensor filed in a diffusion MRI image.
\begin{figure}
  \centering
   \includegraphics[scale=.6]{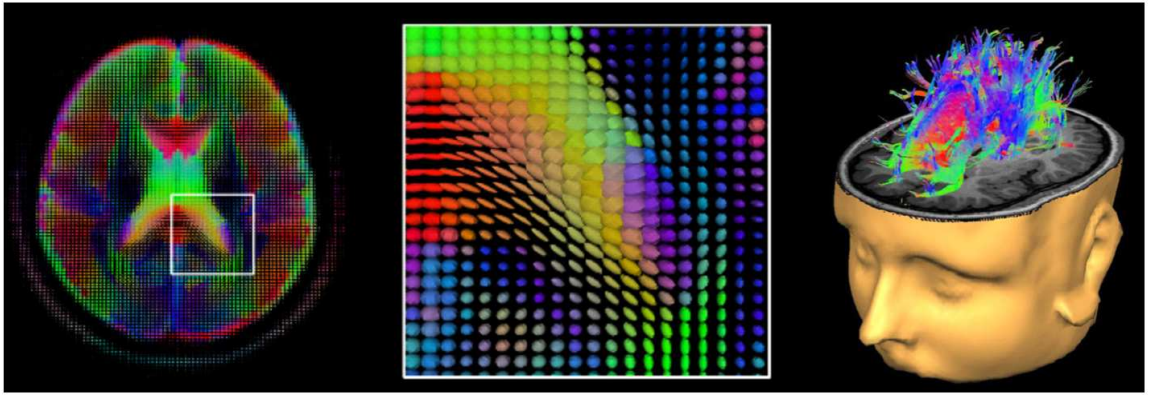}\\
  \caption{Visualization of tensor filed of a brain}\label{brain1}
\end{figure}

In image analysis, the characteristic or precision matrix of the underlying model for tensor observations and distribution of eigenvalues play deterministic roles. Hence, the underlying tensor distribution influences the respective inference. The use of tensor and associated distributional structure in Statistics dates back to McCullagh (1987). McCullagh (1984) had already introduced tensor notation in statistics with particular reference to the computation of polynomial cumulants. For a selective papers about tensors and applications in statistics, we refer to Sakata (2016).

In all pronounced studies in statistical tensor analysis, tensor normal (or multilinear normal) distribution is employed for the underlying distribution of observations. However, a slight change in the specification of the distribution, as pointed by Basser and Pajevic (2003), may play havoc on the resulting inferences. To broaden the scope of the distributions and achieve reasonable inferential conclusions, and in order to  accommodate the heavier tailed distributions in a reasonable way and produce robust inference procedures for applications, tensor t-distribution can be employed in related analysis. From a broader view point, one may define the class of tensor elliptical distributions which includes the latter distribution as special one. In this article, we define a new class of tensor elliptical distributions and study some of its statistical properties.

\setcounter{equation}{0}
\section{Preliminaries}
In this section we introduce related notation to our study and give some definitions. We adhere to the notation of Ohlson et al. (2013).

Let $\X$ be a tensor of order $k$ ($k^{\textnormal{th}}$-order tensor, in tensor parlance), with the dimension $\bp=(p_1,p_2,\ldots,p_k)$ in the $\bx=(x_1,x_2,\ldots,x_k)$ direction. Figure \ref{tensor} shows the special case when $k = 3$. Indeed $2^{\textnormal{nd}}$-order tensor is  matrix, $1^{\textnormal{st}}$-order tensor is vector, and $0^{\textnormal{th}}$-order tensor is scalar.
\begin{figure}
  \centering
   \includegraphics[scale=.5]{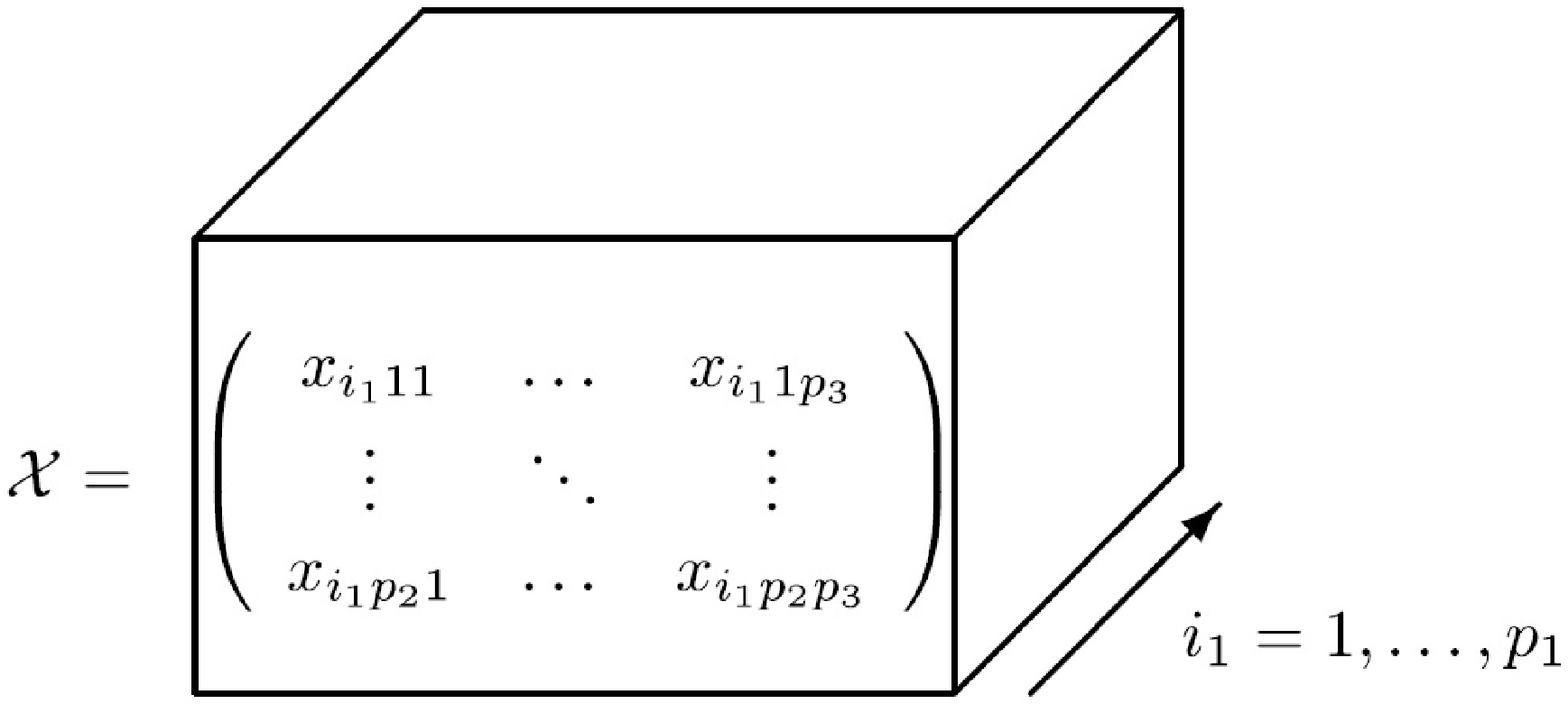}\\
  \caption{Visualization of a 3-dimensional data set as a $3^{\textnormal{rd}}$-order tensor.}\label{tensor}
\end{figure}

In connection with Figure \ref{tensor}, Figure \ref{heart} shows that the collected data can be interpreted as tensor, where the assessment of cardiac ventricular with helical structure is done by DT-MRI.
\begin{figure}
  \centering
   \includegraphics[scale=.2]{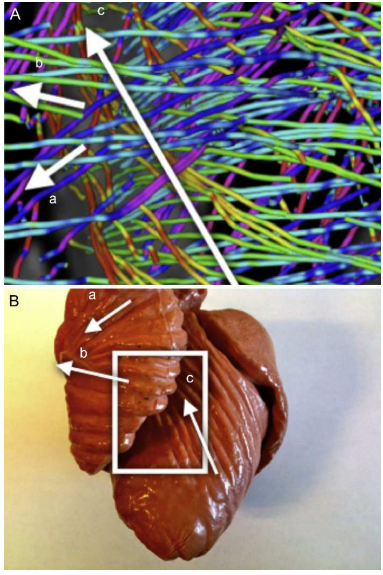}\\
  \caption{Helical structure of the cardiac ventricular anatomy}\label{heart}
\end{figure}

A vectorial representation of a tensor, make the related inference much simpler. Let $\vec\X$ denote the vectorization of tensor $\X=(x_{i_1i_2\cdots i_k})$, according to the definition of Kolda and Bader (2009) given by
\begin{eqnarray}\label{eq11}
\vec\X&=&\sum_{i_1=1}^{p_1}\cdots\sum_{i_k=1}^{p_k}x_{i_1i_2\cdots i_k}\be_{i_1}^1\otimes\cdots\otimes\be_{i_k}^k,\cr
&=&\sum_{I_{\bp}}x_{i_1i_2\cdots i_k}\be_{1:k}^{\bp},
\end{eqnarray}
where $\be_{i_k}^k$, $\be_{i_{k-1}}^{k-1}$, ..., $\be_{i_1}^1$ are the unit basis vectors of size $p_k$, $p_{k-1}$, ..., $p_1$, respectively, $\be_{1:k}^{\bp}=\be_{i_1}^1\otimes\cdots\otimes\be_{i_k}^k$, where $\otimes$ denotes the Kronecker product, $I_{\bp}$ is the index set defined as $I_{\bp}=\{i_1,\ldots,i_k:1\leq i_j\leq p_j, 1\leq j\leq k\}$. In Ohlson et al. (2012), the authors concentrated on the estimation of a Kronecker structured covariance matrix of order three ($k=3$), the so called double separable covariance matrix, generalizing the work of Srivastava et al. (2008), for multilinear normal (MLN) distributions.

Let $\calT^{\bp}$ denote the space of all vectors $\bx=\vec\X$, where $\X$ is a tensor of order $k$, i.e., $\calT^{\bp}=\{\bx:\bx=\sum_{I_{\bp}}x_{i_1i_2\cdots i_k}\be_{1:k}^{\bp}\}$. Note that this tensor space is described using vectors. However, we can define tensor spaces using matrices. This is given in the following definition.
\begin{definition}
Let
\begin{enumerate}
\item[(i)] $\T^{\bp\bq}=\left\{\bX:\bX=\sum_{I_{\bp}\cup I_{\bq}} x_{i_1,\cdots,i_k,j_1,\cdots,j_l}\be_{1:k}^{\bp}(\bd_{1:l}^{\bq})'\right\}$, $I_{\bq}=\left\{j_1,\ldots,j_l:1\leq j_i\leq p_i,\;1\leq i\leq l\right\}$
\item[(ii)] $\T^{\bp\bq}_{\otimes}=\left\{\bX\in\T^{\bp\bq}:\bX=\bX_1\otimes\ldots\otimes\bX_k,\bX_i:p_i\times q_i\right\}$
\item[(iii)] $\T^{\bp}_{\otimes}=\left\{\bX\in\T^{\bp\bp}_{\otimes}:\bX=\bX_1\otimes\ldots\otimes\bX_k,\bX_i:p_i\times p_i\right\}$
\end{enumerate}
\end{definition}

\begin{theorem}\label{MLN} (Ohlson et al., 2013) A tensor $\X$ is MLN of order $k$, denoted by $\X\sim\calN_{\bp}(\bmu,\bSigma)$ if $\bx=\bmu+\bSigma^{\frac{1}{2}}\bu$, where $\bx,\bmu\in\T^{\bp}$, $\bSigma\in\T_{\otimes}^{\bp}$, $\bp=(p_1,\ldots,p_k)$, and the elements of $\bu\in\T^{\bp}$ are independent standard normally distributed.
\end{theorem}
Note that $\bSigma\in\T^{\bp}_{\otimes}$ can be written as Kronecker product $\bSigma=\bSigma_1\otimes\ldots\otimes\bSigma_k$

Indeed, Theorem \ref{MLN} configures the MLN distribution using the stochastic representation of vector $\bx\in\T^{\bp}$. This methodology can be mimicked to extend the above result for elliptical models. Before revealing the main result of this paper, we need to consider the definition of matrix elliptical distributions.

\setcounter{equation}{0}
\section{Tensor Elliptical Distributions}
Let $\bu^{(p^*)}$, $p^*=\prod_{i=1}^k p_i$, denote a random vector distributed uniformly on the unit sphere surface in $\mathbb{R}^{p^*}$, with characteristic function (cf) $\Omega_{p^*}(\cdot)$.
Hereafter, using Theorem 2.2 of Fang et al. (1990), we propose a definition for tensor elliptical (TE) distribution. The methodology behind our definition of TE distribution comes from two facts: (1) a random matrix $\bX$ has matrix elliptical distribution if and only if $\vec\bX$ has vector-variate elliptical distribution which will be used for tensor (see Gupta et al., 2013) (2) the difference between vector-variate elliptical and TE lies in the structure of the parameter space generated by $\bmu$ and $\bSigma$.
\begin{definition}\label{definition of TE}
A random tensor $\X$ is TE of order $k$, denoted by $\X\sim\calE_{\bp}(\bmu,\bSigma,\psi)$, if
\begin{equation}
\bx=\vec(\X)=\bmu+\calR\bSigma^{\frac12}\bu^{(p^*)},
\end{equation}
where $\bx,\bmu\in\T^{\bp}$, $\bSigma^{\frac12}\in\T_{\otimes}^{\bp}$ is any square root, $\bp=(p_1,\ldots,p_k)$, $\calR\geq0$ is independent of $\bu^{(p^*)}$, and $\calR\sim F$, for some cumulative distribution function (cdf) $F(\cdot)$ over $[0,\infty)$, is related to $\psi$ by the following relation
\begin{equation}
\psi(x)=\int_{\mathbb{R}^+}\Omega_{p^*}(xr^2)\textnormal{d}F(r).
\end{equation}
\end{definition}

The question arises whether the parameters in Definition \ref{definition of TE} are uniquely defined. The answer is no. To see this, assume
that $a_i$, $i=1,\ldots,k$ are positive constants such that $a^*=\prod_{j=1}^k a_j$, $\bSigma_j^*=a_j\bSigma_j$, $j=1,\ldots,k$ and $\psi^*(x)=\psi\left(\frac{1}{p^*}x\right)$. Then $\calE_{\bp}(\bmu,\bSigma,\psi)$ and $\calE_{\bp}(\bmu,\bSigma^*,\psi^*)$, where $\bSigma^*=\bSigma_1^*\otimes\ldots\otimes\bSigma_k^*$, define the same tensor elliptical distribution.

Using vector representation, $\bx\in\T^{\bp}$, we can conveniently write the probability distribution function (pdf) of
a TE extending the pdf of MLN distribution. The following result gives the pdf of a random tensor elliptical if it possesses a density, as an extension to Ohlson et al. (2013).
\begin{theorem}\label{theorem of TE pdf} Under the assumptions of Definition \ref{definition of TE}, the pdf of the TE distribution is given by
\begin{equation*}
f_{\X}(\bx)=|\bSigma|^{-\frac12}g\left[(\bx-\bmu)'\bSigma^{-1}(\bx-\bmu)\right],
\end{equation*}
where $g(\cdot)$ is a non-negative function (density generator, say) satisfying
\begin{equation*}
\int_{\mathbb{R}^+}y^{\frac12 p^*-1}g(y)\textnormal{d}y<\infty.
\end{equation*}
We designate $\X\sim\calE_{\bp}(\bmu,\bSigma,g)$.
\end{theorem}
In a similar fashion, we have the following result.
\begin{theorem} Let $\X\sim\calE_{\bp}(\bmu,\bSigma,\psi)$. Then, its characteristic function has form
\begin{equation}
\phi_{\X}(\S)=e^{i\S'\bmu}\psi(\S'\bSigma\S),\quad \S\in\T^{\bp}.
\end{equation}
\end{theorem}
\begin{remark}
Since
\begin{eqnarray*}
|\bSigma|^{-\frac12}&=&|\bSigma_1\otimes\bSigma_2\otimes\ldots\otimes\bSigma_k|^{-\frac12}\cr
&=&\left(|\bSigma_1|^{\frac{p^*}{p_1}}\right)^{-\frac12}\times\left(|\bSigma_2|^{\frac{p^*}{p_2}}\right)^{-\frac12}\times\ldots\times\left(|\bSigma_k|^{\frac{p^*}{p_k}}\right)^{-\frac12}\cr
&=&\prod_{i=1}^k |\bSigma_i|^{-\frac{p^*}{2p_i}}
\end{eqnarray*}
taking $g(y)=(2\pi)^{-\frac{1}{2}p^*}\exp\left(-\frac12 y\right)$ in Definition \ref{definition of TE},
gives the pdf of MLN distribution (as given in Theorem 1 of Ohlson et al., 2013) as
\begin{equation}\label{MLN denisty}
f_{\X}(\bx)=(2\pi)^{-\frac{1}{2}p^*}\left(\prod_{i=1}^k|\bSigma_i|^{-\frac{p^*}{2p_i}}\right)\exp\left[-\frac12(\bx-\bmu)'\bSigma^{-1}(\bx-\bmu)\right],
\end{equation}
where $\bSigma$ is positive definite, $\bx,\bmu\in\T^{\bp}$, $\bSigma\in\T^{\bp}_{\otimes}$, and $p^*=\prod_{i=1}^k p_i$.
\end{remark}
The following result gives the distribution of affine transformations for TE variates.
\begin{theorem} Let $\X\sim\calE_{\bp}(\bmu,\bSigma,\psi)$, with $\vec\X\in\T^{\bp}$, $\bA\in\T^{\bq\bp}$ is nonsingular, and $\bB\in\T^{\bq}$. Then, $\bA\X+\bB\sim\calE_{\bq}(\bA\bmu+\bB,\bA\bSigma\bA')$, where $\bA\bmu+\bB\in\T^{\bq}$ and $\bA\bSigma\bA'\in\T_{\otimes}^{\bq}$.
\end{theorem}
\begin{proof} Let $\by=\bA\bx+\bB$, where $\bx=\vec(\X)$. From the stochastic representation in Definition \ref{definition of TE}, the proof directly follows from
$\by=(\bA\bmu+\bB)+\calR(\bA\bSigma^{\frac12})\bu^{(p^*)}$.
\end{proof}

The following result is a direct consequent of Theorem 2.16 of Gupta et al. (2013) for tensor elliptical distributions.
\begin{theorem} Under the assumptions of Definition \ref{definition of TE}, the pdf of $\calR$ has from
\begin{equation*}
h_{\calR}(r)=\frac{2\pi^{\frac12p^*}}{\Gamma\left(\frac12p^*\right)}r^{p^*-1}g\left(r^2\right), \quad r\geq0.
\end{equation*}
\end{theorem}
The following theorem reveals the distribution of quadratic form for a special case.

\begin{theorem} Let $\X\sim\calE_{\bp}(\0,\bSigma^{(1)},\psi)$,
where $\bSigma^{(1)}=\bSigma\otimes\bI_{p_2}\otimes\ldots\otimes\bI_{p_k}\in\T_\otimes^{\bp},\quad\bp=(p_1,\ldots,p_k)$. Then, the pdf of $\A=\X\X'$ is given by
\begin{equation*}
f(\A)=\frac{\pi^{p^*}|\bSigma|^{-\frac12 p_1}}{\Gamma_{p_1}\left(\frac{1}{2}p^{(1)}\right)}|\A|^{\frac{1}{2}p^{(1)}-p_1-1}g(\tr\bSigma^{-1}\A)
\end{equation*}
where $p^{(1)}=\prod_{j=2}^k p_j$.
\end{theorem}

In the forthcoming section, we provide a weighting representation of the pdf of TE variate using the Laplace operator.

\setcounter{equation}{0}
\section{Weighting Representation}
Although the proposed theorems in previous section are obtained conventionally, it is not easy to achieve other statistical properties of the TE distributions from Definition \ref{definition of TE} straightforwardly. However, under a mild conditions, one can make connection between densities of TE and MLN pdfs and derive other properties of the TE distributions using MLN distributions. In this section, we propose a weighting representation which connects densities of the TE and MLN distributions. This result is given in the following theorem.
\begin{theorem}\label{weighting representation} Let $\X\sim\calE_{\bp}(\bmu,\bSigma,g)$, where $\bmu\in\T^{\bp}$, $\bSigma\in\T_\otimes^{\bp}$ and $g:\mathbb{R}^+\rightarrow\mathbb{R}^+$. Also assume that $g(s^2)$ is differentiable when $s^2$ is sufficiently large, and $g(s^2)$ vanishes faster than $s^{-k}$; $k > 1$ as $s\to\infty$. Then, the pdf of $\X$ can be represented as an integral of series of MLN pdfs given by
\begin{equation*}
f_{\X}(\bx)=\int_{\mathbb{R}^+}\calW(t)f_{\calN_{\bp}(\bmu,t^{-1}\bSigma)}(\bx)\textnormal{d}t,
\end{equation*}
where $f_{\calN_{\bp}(\bmu,t^{-1}\bSigma)}(\cdot)$ is the pdf of $\calN_{\bp}(\bmu,t^{-1}\bSigma)$ and $\calW(\cdot)$ is a weighting function.
\end{theorem}
\begin{proof} Let $s^2=\frac12(\bx-\bmu)'\bSigma^{-1}(\bx-\bmu)$ and
\begin{equation*}
\calW(t)=(2\pi)^{\frac{1}{2}p^*}t^{-\frac{p^*}{2}}\calL^{-1}\left[g\left( 2s^2\right)\right],
\end{equation*}
where $\calL$ is the Laplace transform operator. It should be noted that under the regularity condition on $g(s^2)$, the inverse Laplace transform exists. Then, from $|\bSigma|^{-\frac12}=\prod_{i=1}^k |\bSigma_i|^{-\frac{p^*}{2p_i}}$, we have
\begin{eqnarray*}
f(\bx)=|\bSigma|^{-\frac12}g(2s^2)&=&|\bSigma|^{-\frac12}\calL\left[\calW(t)(2\pi)^{-\frac{1}{2}p^*}t^{\frac{p^*}{2}}\right]\cr
&=&\calL\left[\calW(t)(2\pi)^{-\frac{1}{2}p^*}t^{\frac{p^*}{2}}\prod_{i=1}^k|\bSigma_i|^{-\frac{p^*}{2p_i}}\right]\cr
&=&\int_{\mathbb{R}^+}\calW(t)(2\pi)^{-\frac{1}{2}p^*}t^{\frac{p^*}{2}}\left(\prod_{i=1}^k|\bSigma_i|^{-\frac{p^*}{2p_i}}\right)e^{-ts^2}\textnormal dt\cr
&=&\int_{\mathbb{R}^+}\calW(t)(2\pi)^{-\frac{1}{2}p^*}\left(\prod_{i=1}^k|t^{-\frac1k}\bSigma_i|^{-\frac{p^*}{2p_i}}\right)
e^{-\frac12(\bx-\bmu)'(t^{-1}\bSigma)^{-1}(\bx-\bmu)}\textnormal dt\cr
&=&\int_{\mathbb{R}^+}\calW(t)f_{\calN_{\bp}(\bmu,t^{-1}\bSigma)}(\bx)\textnormal dt.
\end{eqnarray*}
The proof is complete.
\end{proof}

Thus, a TE variable is an integral over all MLN variables having the same covariance subject to different scales.

Since $f_{\X}(\cdot)$ is the pdf of $\X$, using Fubini's theorem, we obtain
\begin{eqnarray}
1=\int_{\chi} f_{\X}(\bx)\textnormal d\bx&=&\int_\chi\int_{\mathbb{R}^+}\calW(t)f_{\calN_{\bp}(\bmu,t^{-1}\bSigma)}(\bx)\textnormal{d}t\textnormal d\bx\cr
&=&\int_{\mathbb{R}^+}\calW(t)\int_\chi f_{\calN_{\bp}(\bmu,t^{-1}\bSigma)}(\bx)\textnormal d\bx\textnormal{d}t\cr
&=&\int_{\mathbb{R}^+}\calW(t)\textnormal{d}t
\end{eqnarray}
where $\chi$ is the sample space. Hence, for positive weighting functions $\calW(\cdot)$, the weighting representation of TE distributions can be interpreted as an scale mixture of MLN distributions. However, sometimes, $\calW(\cdot)$ can be negative. Note that a TE distribution is completely defined by the matrix $\bSigma\in\T^{\bp}_{\otimes}$ and the scalar weighting function $\calW(\cdot)$.

Theorem \ref{weighting representation} enables us to describe more properties of TE distributions via MLN distributions. This can be done using the following important result.
\begin{theorem}
Let $\bx\sim\calE_{\bp}(\bmu,\bSigma,g)$, $\bmu\in\T^{\bp}$, $\bSigma\in\T_\otimes^{\bp}$ and $g:\mathbb{R}^+\rightarrow\mathbb{R}^+$ with weighting function $\calW(\cdot)$, and $B(\bx)$ be any Borel measurable function of $\bx\in\T^{\bp}$. Then, if $E[B(\bx)]$ exists, we have
\begin{equation*}
E[B(\bx)]=\int_{\mathbb{R}^+}\calW(t)E_{\calN_{\bp}(\bmu,t^{-1}\bSigma)}[B(\bx)]\textnormal dt
\end{equation*}
\end{theorem}

\setcounter{equation}{0}
\section{Examples}
In this section, we provide some examples of TE distributions based on Definition \ref{definition of TE} with respective weighting function, as defined in Theorem \ref{weighting representation}.

Firstly, we consider some examples in which the weighting function $\calW(.)$
is always positive, resulting to scale mixture of multilinear normal distributions.
\begin{enumerate}
\item[(i)] Multilinear normal distribution (Ohlson et al., 2013)\\
The weighting function has form
\begin{equation*}
\mathcal{W}(t)=\delta (t-1),
\end{equation*}%
where $\delta (\cdot )$ is the dirac delta or impulse function
having the property $\int\limits_{\mathbb{R}}f(x)\delta
(x)dx=f(0)$, for every Borel-measurable function $f(\cdot )$.
\item[(ii)]Multilinear $\protect\varepsilon $-contaminated
normal distribution\\
We say the random tensor $\X\in\T^{\bp}$ has multilinear
$\varepsilon $-contaminated normal
distribution
if it has the following density
\begin{align*}
f_{\X}(\bx)=& \frac{1}{(2\pi)^{\frac{1}{2}p^*}}\left(\prod_{i=1}^k|\bSigma_i|^{-\frac{p^*}{2p_i}}\right)\bigg\{(1-\varepsilon )\exp \left[
-\frac{1}{2}(\bx-\bmu)'\bSigma^{-1}(\bx-\bmu) \right] \\
& +\frac{\varepsilon }{\sigma ^{p^*}}\exp \left[ -\frac{1}{2\sigma ^{2}}%
(\bx-\bmu)'\bSigma^{-1}(\bx-\bmu) \right] \bigg\}.
\end{align*}%
Then it can be concluded that the weighting function is given by
\begin{equation*}
\mathcal{W}(t)=(1-\varepsilon )\delta (t-1)+\varepsilon \delta
(t-\sigma ^{2}).
\end{equation*}
\item[(iii)]Tensor $t$-distribution\\
We say the random matrix $\X\in\T^{\bp}$ has
tensor $t$-distribution
if it has the following density
\begin{eqnarray}
f_{\X}(\bx)=\frac{\nu^{\frac{p^*}{2}}\Gamma\left(\frac{p^*+\nu}{2}\right)}{%
\pi^{\frac{1}{2}p^*}\Gamma(\frac{\nu}{2} )}\left\{ 1+\frac{1}{\nu }%
(\bx-\bmu)'\bSigma^{-1}(\bx-\bmu) \right\}
^{-(p^*+\nu )}. 
\label{Matrix variate q t}
\end{eqnarray}%
The corresponding weighting function has form
$\mathcal{W}(t)=\frac{\left(\frac{t\nu}{2}
\right)^{\frac{\nu}{2} }e^{-\frac{t\nu}{2}
}}{t\Gamma\left(\frac{\nu}{2} \right)}$.

The tensor Cauchy distribution is obtained by setting $%
\nu =1$ in (\ref{Matrix variate q t}).

It is of much interest to consider cases in which
the weighting function $\calW(.)$ is not always positive. Such kind
of distributions are not scale mixture of multilinear normal
distributions. The item below is not a tensor distribution, however it is $0^{\textnormal{th}}$-order tensor distribution.
\item[(iv)] The one-dimensional distribution with the following
density
\begin{eqnarray*}
f(x)=\frac{\sqrt{2}}{\pi\sigma}\left[1+\left(\frac{x}{\sigma}\right)^{4}\right]^{-1},
\end{eqnarray*}
where the weighting function is given by
$\calW(t)=\frac{1}{\sqrt{t\pi}}\;\sin\left(\frac{t}{2}\right)$.
\end{enumerate}

\setcounter{equation}{0}
\section{Inference}

\begin{theorem}
Suppose that tensor variables $\X_1,\ldots,\X_n$ are jointly distributed with the following pdf
\begin{equation*}
\prod_{i=1}^k |\bSigma_i|^{-\frac{p^*}{2p_i}}g\left(\sum_{j=1}^n \bx_j'\bSigma^{-1}\bx_j\right),\quad \bSigma=\bSigma_1\otimes\bSigma_2\otimes\ldots\otimes\bSigma_k
\end{equation*}
such that $\sigma_{p_2p_2}^{(2)}=\sigma_{p_3p_3}^{(3)}=\ldots=\sigma_{p_kp_k}^{(k)}=1$, where $\bSigma_r=\left(\sigma_{ij}^{(r)}\right)$. Further, suppose $g(\cdot)$ is such that $g(\bx'\bx)$ is a pdf in $\mathbb{R}^{p^*}$ and $y^{p^*/2}g(y)$ has a finite positive maximum $y_g$ (see Anderson et al., 1986 for the existence of $y_g$). Suppose that $\tilde\bSigma$ is an estimator which obtains from solving the following equations (see Ohlson et al., 2013)
\begin{eqnarray*}
\tilde\bSigma_1&=&\frac{1}{p^*_{2:k}n}\sum_{j=1}^n\bx_j'\bSigma^{-1}_{2:k}\bx_j\cr
&&\mbox{and}, \ \mbox{for} \ r=2,\ldots,k\cr
\tilde\bSigma_r&=&\frac{1}{p^*_{1:r-1}p^*_{r+1:k}n}\sum_{j=1}^n{\bx_i^{2,r(r)}}'\left(\bSigma^{2,r}_{1:k\setminus r}\right)^{-1}\bx_i^{2,r(r)},
\end{eqnarray*}
where
\begin{eqnarray*}
\bx_j^{2,r(r)}&=&\sum_{l_{\bp}}x_{i_1,\ldots i_k}\be^{2,r}_{i_1:i_k\setminus i_r}{\be_{i_r}^{p_r}}'\cr
\bSigma^{2,r}_{1:k\setminus r}&=&\bSigma_2\otimes\ldots\otimes\bSigma_{r-1}\otimes\bSigma_1\otimes\ldots\otimes\bSigma_k\cr
\be^{2,r}_{i_1:i_k\setminus i_r}&=&\be_{i_1}^{p_1}\otimes\be_{i_3}^{p_3}\otimes\ldots\otimes\be_{i_{r-1}}^{p_{r-1}}\otimes\be_{i_{r+1}}^{p_{r+1}}
\otimes\ldots\otimes\be_{i_k}^{p_k}
\end{eqnarray*}

Then, the MLE of $\bSigma$ is given by
\begin{equation*}
\hat\bSigma=\frac{p^*}{y_g}\tilde\bSigma
\end{equation*}

\end{theorem}
\begin{proof}
Let $\bA=|\bSigma|^{-\frac{1}{p^*}}\bSigma$.
Also for any $j=1,\ldots,n$ write
\begin{equation}\label{eq61}
d_j=\bx_j'\bSigma^{-1}\bx_j=|\bSigma|^{-\frac{1}{p^*}}\bx_j'\bA^{-1}x_j.
\end{equation}
Since $|\bSigma|^{-\frac12}=\prod_{i=1}^k |\bSigma_i|^{-\frac{p^*}{2p_i}}$, the likelihood can be written as
\begin{eqnarray}\label{eq62}
\calL&=&|\bSigma|^{-\frac12}g\left(\sum_{j=1}^n d_j\right)\cr
&=&\left(|\bSigma|^{\frac{1}{p^*}}\right)^{-\frac{p^*}{2}}\left(\sum_{j=1}^n d_j\right)^{\frac{p^*}{2}}g\left(\sum_{j=1}^n d_j\right)\cr
&=&\left(\sum_{j=1}^n x_j'\bA^{-1}x_j\right)^{-\frac{p^*}{2}}d^{\frac{p^*}{2}}g(d),
\end{eqnarray}
where $d=\sum_{j=1}^n d_j$.

The maximum of \eqref{eq62} is attained at $\hat\bA=\tilde\bA$ and $\hat d=y_g$. Then the MLE of $\bSigma$ is given by
\begin{equation}\label{eq63}
\hat\bSigma=|\hat\bSigma|^{\frac{1}{p^*}}\hat\bA=\frac{|\hat\bSigma|^{\frac{1}{p^*}}}{|\tilde\bSigma|^{\frac{1}{p^*}}}\tilde\bSigma.
\end{equation}
On the other hand, from \eqref{eq61} we get
\begin{eqnarray}\label{eq64}
|\hat\bSigma|^{\frac{1}{p^*}}&=&\frac{\sum_{j=1}^n \bx_j'\hat\bA^{-1}\bx_j}{\sum_{j=1}^n \hat d_j}=
\frac{\sum_{j=1}^n \bx_j'\tilde\bA^{-1}\bx_j}{\hat d}=\frac{\sum_{j=1}^n \bx_j'\tilde\bA^{-1}\bx_j}{y_g}\cr
|\tilde\bSigma|^{\frac{1}{p^*}}&=&\frac{\sum_{j=1}^n \bx_j'\tilde\bA^{-1}\bx_j}{\sum_{j=1}^n \tilde d_j}=
\frac{\sum_{j=1}^n \bx_j'\tilde\bA^{-1}\bx_j}{\tilde d}=\frac{\sum_{j=1}^n \bx_j'\tilde\bA^{-1}\bx_j}{p^*}
\end{eqnarray}
Substituting \eqref{eq64} in \eqref{eq63} and using Theorem 4.1 of Ohlson et al. (2013) gives the result.
\end{proof}

\section{Conclusion}
In this article, for the purpose of robust inferring on diffusion tensor magnetic resonance imaging (DT-MRI) observations, we proposed a class of tensor elliptical (TE) distributions. This class includes many heavier tail distributions than the tensor normal or multilinear normal (MLN) distribution. Important statistical properties including the characteristic function along with the distribution of affine transformations derived. A weighting representations also exhibited that connects densities of TE and MLN distributions.

\section*{References}
\baselineskip=12pt
\def\ref{\noindent\hangindent 25pt}

\ref Basser, P. J. and Pajevic, S. (2003). A normal distribution for tensor-valued random variables: applications to diffusion tensor MRI. {\em IEEE Transactions on Medical Imaging}, 22(7):785-794.

\ref Fang, K.T., Kotz, S., Ng, K.W. (1990) {\em Symmetric Multivariate and Related Distributions}. Chapman and
Hall, London.


\ref Gupta, A.K. and Varga, T., and Bondar, T. (2013) {\em Elliptically Contoured Models in Statistics and Portfolio Theory}, 2rd Ed., Springer,
New York.
\ref Kolda, T.G. and Bader, B.W. (2009). Tensor decompositions and applications. {\em SIAM Review}, 51(3):455-500.

\ref McCullagh, P. (1987). {\em Tensor Methods in Statistics}, Chapman \& Hall, London.

\ref McCullagh, P. (1984). Tensor notation and cumulants of polynomials. {\em Biometrika}, 71(3):461-476.

\ref Ohlson, M., Rauf Ahmad, M., and von Rosen, D. (2012). More on the Kronecker structured covariance matrix. {\em Communications in Statistics Theory and Methods}, 41:2512-2523.

\ref Ohlson, M., Rauf Ahmad, M., and von Rosen, D. (2013). The multilinear normal distribution: Introduction and some basic properties. {\em Journal of Multivariate Analysis}, 113:37-47.

\ref Sakata, T. (2016). {\em Applied Matrix and Tensor Variate Data Analysis}, Springer, Japan.

\ref Srivastava, M., von Rosen, T., and von Rosen, D. (2008). Models with a Kronecker product covariance structure: estimation and testing. {\em Mathematical Methods in Statistics}, 17:357-370.

\end{document}